\documentclass[11pt]{amsart}

\usepackage{latexsym}
\usepackage{amsfonts}

\def\cal{\mathcal}

\newtheorem{defi}{Definition}[section]

\newtheorem{lem}[defi]{Lemma}
\newtheorem{theo}[defi]{Theorem}
\newtheorem{cor}[defi]{Corollary}
\newtheorem{pr}[defi]{Proposition}
\newtheorem{re}[defi]{Remark}

\font\tenmsy=msbm10

\def\Bbb#1{\hbox{\tenmsy#1}} 

\title[On the group of automorphisms of a quasi-affine variety]{On the group of automorphisms of a quasi-affine variety}
\makeatletter

\@addtoreset{equation}{section}
\makeatother
\author{Zbigniew Jelonek}
\address[Z. Jelonek]{Instytut Matematyczny\\
Polska Akademia Nauk\\
\'Sniadeckich 8, 00-956 Warszawa, Poland}
\email{najelone@cyf-kr.edu.pl}

 \subjclass{14
R 10}
\thanks{The author was partially supported by the grant of NCN,  2014-2017. }

\date{\today}

\begin{document}

\maketitle

\begin{abstract}
Let $\Bbb K$ be an algebraically closed field of characteristic
zero. We show that if the automorphisms group of a quasi-affine
variety $X$ is infinite, then $X$ is uniruled.
\end{abstract}

\bibliographystyle{alpha}
\maketitle

\section{Introduction.}

Automorphism groups of an open varieties have always attracted a
lot of attention, but the nature of this groups is still not
well-known. For example the group of automorphisms of $\Bbb K^n$
is understood only in the case $n=2$ (and $n=1$, of course).
 Let $Y$ be a an open variety. It is natural to ask
when the group of automorphisms of $Y$ is finite. We gave a
partial answer to this questions in our papers \cite{jel1},
\cite{jel2}, \cite{jel3} and \cite{jel4}. In \cite{ii2} Iitaka
proved that $Aut(Y)$ is finite if $Y$ has a maximal logarithmic
Kodaira dimension. Here we focus on the group of automorphisms of
an affine or more generally  quasi-affine variety over an
algebraically closed field of characteristic zero. Let us recall
that a quasi-affine variety is an open subvariety of some affine
variety. We prove the following:

\begin{theo}
Let $X$ be a quasi-affine (in particular affine)  variety over an
algebraically closed field of characteristic zero. If $Aut(X)$ is
infinite, then $X$ is uniruled, i.e., $X$ is covered by rational
curves.
\end{theo}

This generalizes our old results from \cite{jel3} and \cite{jel4}.
Our proof uses in a significant way a recent progress in a minimal
model program ( see \cite{bir}, \cite{bir1}, \cite{p-s}) and is
based on our old ideas from \cite{jel1}, \cite{jel2}, \cite{jel3}
and \cite{jel4}.

\vspace{5mm} {\it Acknowledgements.} The author is grateful to
Professor Y. Prokhorov from Steklov Institute of Mathematics for
helpful discussions.

\section{Terminology.}

We assume that ground field $\Bbb K$ is algebraically closed of
characteristic zero.  For an algebraic variety $X$ (variety is
here always irreducible) we denote by $Aut(X)$ the group of all
regular automorphisms of $X$ and by $Bir(X)$ the group of all
birational transformations of $X$. By $Aut_{1}(X)$ we mean the
group of all birational transformation which are regular in
codimension one, i.e., which are regular isomorphisms outside
subsets of codimension at least two.   If $X\subset \Bbb P^n(\Bbb
K)$ then we put $Lin(X)=\{ f\in Aut(X): f = res_X T , T\in
Aut(\Bbb P^n(\Bbb K))\}.$ Of course, the group $Lin(X)$ is always
an affine group.

Let $f:X -\rightarrow Y$ be a rational mapping between projective
normal varieties. Then $f$ is determined outside some closed
subset $F$ of codimension at least two. If $S\subset X$ and
$S\not\subset F$ then by $f(S)$ we mean the set $f(X\setminus F).$
Similarly for $R\subset Y$ we will denote the set $\{ x\in
X\setminus F: f(x)\in R\}$ by $f^{-1}(R)$.

If $f:X -\to Y$ is a birational mapping and the mapping $f^{-1}$
does not contract any divisor, we say that $f$ is a birational
contraction.


An algebraic irreducible variety $X$ of dimension $n>0$ is called
uniruled if there exists an irreducible variety $W$ of dimension
$n-1$ and a rational dominant mapping $\phi : W\times \Bbb
P^1(\Bbb K) \  - \rightarrow X$. Equivalently: an algebraic
variety $X$ is uniruled if and only if for a generic point in X
there exists a rational curve in $X$ through this point.

 We say that divisor $D$
is \Bbb Q-Cartier if for some non-zero integer $m\in \Bbb Z$ the
divisor $mD$ is Cartier. If every divisor on $X$ is $\Bbb
Q$-Cartier then we say that $X$ is \Bbb Q-factorial.

In this paper  we treat a hypersurface $H=\bigcup^r_{i=1}
H_i\subset X$ as a reduced divisor $\sum^r_{i=1} H_i$, and
conversly a reduced divisor will be treated as a hyperserface.

\section{Weil divisors on a normal variety.}

In this section we recall  (with suitable modifications) some
basic results about divisors on a normal variety (see e.g.,
\cite{ii1}).

\begin{defi}
Let $X$ be a normal projective variety. We will  denote by
$Div(X)$ the group of all Weil divisors on $X$. For $D\in Div(X)$
the set of all Weil divisors linearly equivalent to $D$, is called
a complete linear system given by $D$ and it will be denoted by $|
D |$. Moreover we set  $L(D):=\{ f\in \Bbb K (X) : f=0 \ or \
D+(f)\geq 0 \}.$
\end{defi}

We have the following (e.g., \cite{ii1}, 2.16, p.126)

\begin{pr}
If $D$ is an effective divisor on a normal complete variety $X$
then $L(D)$ is a finite dimensional vector space (over $\Bbb K$).
\end{pr}

\begin{re}
{\rm The set $|D |$ (if non-empty) has a natural structure of a
projective space of dimension dim  $L(D) - 1.$ By a basis of $|D|$
we mean any subset $\{D_0,...,D_n\}\subset |D|$ such that
$D_i=D+(\phi_i)$ and $\{ \phi_0,...,\phi_n\}$ is a basis of
$L(D).$}
\end{re}

Let us recall the next

\begin{defi}
If $D$ is an effective Weil divisor on a normal complete variety
$X$ then by a canonical mapping given by $|D|$ and a basis $\phi$,
we mean the mapping $i_{(D,\phi)} =(\phi_0:...:\phi_n) :
X\rightarrow \Bbb P^n(\Bbb K)$,   where
$\phi=\{\phi_0,...,\phi_n\}\subset L(D)$ is a basis of $L(D).$
\end{defi}

Let $X$ be a normal variety and $Z$ be a closed subvariety of $X$.
Put $X'=X\setminus Z.$ We would like to compare the group $Div(X)$
and $Div(X').$ It can be easily checked that the following
proposition is true (compare \cite{har1}, 6.5., p. 133):

\begin{pr}
Let $j_{X'}: Div(X)\ni \sum_{i=1}^r n_iD_i\rightarrow \sum_{i=1}^r
n_i (D_i\cap X')\in Div(X').$ Then $j_{X'}$ is an epimorphism.
Moreover it preserves the linear equivalence. If additionally
$codim \ Z\geq 2$ then $j_{X'}$ is an isomorphism.
\end{pr}

Now we  define the pull-back of a  divisor under a rational map
$f:X -\to Y.$ Recall that  a Cartier divisor can be given by a
system $\{ U_\alpha, \phi_\alpha\}$ where $\{U_\alpha\}$ is a some
open covering of $X$, $\phi_\alpha\in \Bbb K (U_\alpha)$ and
$\phi_\alpha/\phi_\beta\in {\cal O}^*(U_\alpha\cap U_\beta).$

\begin{defi}
Let $f:X\rightarrow Y$ be a  dominant morphism between normal
varieties. Let $D$ be a Cartier divisor on $Y$ given by a system
$\{ U_\alpha, \phi_\alpha\}$.  By the pullback of the divisor $D$
by $f$ we mean the divisor $f^*D$ given by a system $\{
f^{-1}(U_\alpha), \phi_\alpha\circ f\}.$ More generally if $f$ is
a rational  map and $X_f$ denotes the domain of $f$ we put
$f^*(D):={j_{X_f}}^{-1}(res_{X_f} f)^*D.$

Finally let $f$ be as above and $D$ be an arbitrary Weil divisor
on $Y$. Let us assume additionally that  {\rm codim}
$f^{-1}(Sing(Y))\ge 2.$  Then we have a regular map $f:
X_f\setminus W\rightarrow Y_{reg}$ (where $W:=f^{-1}(Sing(Y))$ and
we put $f^*D:={j_{X_f\setminus W}}^{-1}f^*({j_{Y_{reg}}}(D)).$
\end{defi}

By a simple verification we have

\begin{pr}
Let $f:X -\rightarrow Y$ be a dominant rational mapping between
complete normal varieties, such that $f^{-1}(Sing(Y))$ has
codimension at least two. Then $f^*:Div(Y)\ni D \rightarrow
f^*D\in Div(X)$ is a well-defined homomorphism, which preserves
the linear equivalence. Moreover,
$Supp(f^*(D))=\overline{f^{-1}(Supp(D))}.$$\square$
\end{pr}

\begin{cor}\label{izo}
Let $f$ be as above. Let us assume additionally that $f$ is an
isomorphisms in codimension one. Then $f^*: Div(Y)\rightarrow
Div(X)$ is an isomorphisms which preserves the linear equivalence.
$\square$
\end{cor}

Finally we have the following important result:

\begin{pr}
Let $X$ be a normal complete variety and $f\in Aut_{1}(X)$. Let
$D$ be an effective divisor on $X$ and  $f^*D'=D$. Then ${\rm dim}
|D|= {\rm dim} \  |D'|:=n$ and there exists a unique automorphism
$T(f)\in Aut(\Bbb P^n(\Bbb K))$ such that the folowing diagram
commutes

\begin{center}
\begin{picture}(240,160)(-40,40)
\put(5,180){\makebox(0,0)[tl]{$X$}}
\put(178,180){\makebox(0,0)[tl]{$X$}}
\put(0,40){\makebox(0,0)[tl]{$\Bbb P^n (\Bbb K)$}}
\put(170,40){\makebox(0,0)[tl]{$\Bbb P^n (\Bbb K)$}}
\put(80,50){\makebox(0,0)[tl]{$T(f)$}}
\put(85,190){\makebox(0,0)[tl]{$f $}}
\put(15,110){\makebox(0,0)[tl]{$i_{D}$}}
\put(185,110){\makebox(0,0)[tl]{$i_{D'}$}}
\multiput(25,175)(8,0){17}{\line(1,0){5}}
\put(157,175){\vector(1,0){10}} \put(35,35){\vector(1,0){130}}
\multiput(180,165)(0,-8){14}{\line(0,-1){5}}
\put(180,53){\vector(0,-1){10}}
\multiput(10,165)(0,-8){14}{\line(0,-1){5}}
\put(10,53){\vector(0,-1){10}}
\end{picture}
\end{center}

\vspace{3mm}

\end{pr}

\begin{proof} First of all let us note that $T(f)$ if exists it is unique.
Further, by Corollary \ref{izo} we have $f^*(|D'|)= |D|$ and $f^*$
transforms any basis of $|D'|$ onto a basis of $|D|.$ Let $\phi$
and $\psi$ be suitable bases such that $i_D=i_{(D,\phi)}$ and
$i_{D'}=i_{(D',\psi)}.$

We have $i_{D'}\circ f = (\psi_0,...,\psi_n)\circ f.$ But
$f^*(D'+(\psi_i))=f^*(D')+f^*(\psi_i)=D+(\psi_i\circ f).$ It means
that rational functions $(\psi_i\circ f), i=0,...,n$ are the basis
of $L(D).$ Hence there exists a non-singular matrix ${[a_{i,j}]}$
such that $\psi_i\circ f= \sum_{j=0}^n a_{i,j} \phi_j.$ Now it is
clear that it is enough to take as $T(f)$ the projective
automorphism of $\Bbb P^n(\Bbb K)$ given by the matrix
$[a_{i,j}].$ \end{proof}

\begin{cor}\label{natural}
Let $G$ be a  subgroup of $Aut_1(X)$ such that $G^*D =D$ for some
effective divisor $D.$ Let us denote $\overline{i_D(X)}=X'\subset
\Bbb P^n(\Bbb K),$ $n={\rm dim} \ |D|.$ Then there is a natural
homomorphism $T:G\rightarrow Lin(X')$. Moreover, if $D$ is very
big (i.e., the mapping $i_D$ is a birational embedding) , then $T$
is a monomorphism.
\end{cor}

\begin{proof}
It is enough to take above $D'=D$ and $\phi=\psi.$ The last
statement is obvious.
\end{proof}

\begin{re}
{\rm In our application we deal only with normal $\Bbb
Q-$factorial varieties. Hence we can restrict here only to $\Bbb
Q-$Cartier divisors. However the author thinks that the language
of Weil divisors is more natural here.}
\end{re}

\section{Varieties with good covers.}
We begin this section by recalling  the definition of a big
divisor ( see \cite{k-m}, p. 67):

\begin{defi}
Let $X$ be a projective $n$-dimensional variety and  $D$ a Cartier
divisor on $X.$ The divisor  $D$ is called big if {\rm dim}
$H^0(X, {\cal O}_X(kD))> ck^n$ for some $c>0$ and $k>>1.$
\end{defi}

If $f:X\to Y$ is a birational morphisms and $D$ is a big (Cartier)
divisor, then also its pullback $f^*(D)$ is big. Indeed the line
bundle ${\cal O}_X(mf^*(D))=f^*{\cal O}_Y(mD)$ has at least as
many sections as the bundle ${\cal O}_Y(mD)$. We show later that
it is also true for suitable birational mappings (see Lemma
\ref{big1}). We have a following characterization of big divisors
( see \cite{k-m}, Lemma 2.60, p. 67):

\begin{pr}\label{big}
Let $X$ be a projective $n$-dimensional variety and  $D$ a Cartier
divisor on $X.$ Then the following are equivalent:

1) $D$ is big,

2) $mD\sim A+E$, where $A$ is ample and $E$ is effective Cartier
divisor,

3) for  $m>>0$ the rational map $\iota_{mD}$ associated with the
system $|mD|$ is a birational embedding,

4) the image of $\iota_{mD}$ has dimension $n$ for $m>>0.$
\end{pr}

In the sequel we need the following observation:

\begin{lem}\label{big}
Let $X$ be a smooth projective variety and let $D=\sum^r_{i=1} a_i
D_i$ be a big divisor on $X$. Then $Supp(D)=\sum^r_{i=1} D_i$ is
also a big divisor on $X.$
\end{lem}

\begin{proof}
Let $a=max_ {i=1,...,r} \{ a_i\}$ and $b_i=a-a_i.$ The divisor
$E=\sum^r_{i=1} b_i D_i$ is effective. By Proposition \ref{big},
p. 2) the divisor $D+E=a Supp(D)$ is also big. Hence we  conclude
by 3) of Proposition \ref{big}.
\end{proof}

\begin{defi}
Let $X$ be a normal projective variety and let $D$ be a Weil
divisor on $X$. We say that $X$ is very big if the rational map
$\iota_{D}$ associated with the system $|D|$ is a birational
embedding. We say that $D$ is big if for some $m\in \Bbb N$ the
divisor $mD$ is very big.
\end{defi}

We have the following simple lemma:

\begin{lem}\label{big1}
Let $X, Y$ be normal projective varieties and let $D$ be an
effective
 big divisor on $Y$. Let $\phi: X-\to Y$ be a birational mapping such that {\rm codim} $\phi^{-1}(Sing(Y))\ge 2.$
Then the divisor $\phi^*(D)$ is also  big.
\end{lem}

\begin{proof}
It is enough to assume that $D$ is very big and prove that then
$\phi^*(D)$ is also very big. Take $f_0=1$ and let divisors $\{
D+(f_0), D+(f_1),..., D+(f_s)\},$ where $f_i\in \Bbb K(Y),
i=0,..., s$ form a basis of a system $|D|.$ By the assumption the
regular mapping $\Psi: Y\setminus Supp(D)\ni x\mapsto (f_1(x),...,
f_s(x))\in \Bbb K^s$ is a birational morphism. The system
$|\phi^*(D)|$ contains divisors $\{ \phi^*(D), \phi^*(D)+(f_1\circ
\phi),..., \phi^*(D)+(f_s\circ \phi)\}.$ Since rational functions
$1, f_1\circ \phi,..., f_s\circ \phi$ are linearly independent, we
can extend  them to some basis $B$ of $L(\phi^*(D)).$  Let
$\Psi':X\setminus |Supp(\phi^*(D))|\to \Bbb K^N$ be a mapping
given by a system $|\phi^*(D)|$ and the basis $B.$ The mapping
$\Psi'$ after composition with a suitable projection $\Bbb K^N\to
\Bbb K^s$, is equal to $\Psi\circ \phi.$ Since the latter mapping
is birational, the mapping $\Psi'$ is also birational.
\end{proof}

We shall use:

\begin{defi}\label{good}
Let $X$ be an (open) variety.  We say that $X$ has a good cover
$Y$, if there exists a completion $\overline{X}$ of $X$ and a
smooth projective variety $Y$ with a birational morphism $g: Y \to
\overline{X}$ such that:

1) $D:=g^{-1}(\overline{X}\setminus X)$ is a big hypersurface in
$Y$,

2) $Aut(X)\subset Aut(Y\setminus D)$, i.e., every automorphism of
$X$ can be lifted to an automorphism of $Y\setminus D.$
\end{defi}

Our next aim is to show that quasi-affine varieties have good
covers.

\begin{pr}
Let $X$ be a quasi-affine variety.Then $X$ has a good cover.
\end{pr}

\begin{proof}
By the assumption there is an affine variety $X_1$ such that
$X\subset X_1$ is an open dense subset. Since  $X_1$ is affine, we
can assume  that it is a closed subvariety of  some $\Bbb K^N.$
Denote by $\overline{X}$ the projective closure of $X_1$ in $\Bbb
P^N.$ Let $\pi_\infty$ be the hyperplane at infinity in $\Bbb P^N$
and $V:=\overline{X}.\pi_\infty$ be a divisor at infinity on
$\overline{X}.$ Of course $V$ is a big (even very ample)  Cartier
divisor.

Let $ h: Y \to \overline{X}$ be a canonical desingularization of
$\overline{X}$ ( see e.g., \cite{kol}, \cite{wlo}). Then
$h_{|h{-1}(X)}: h^{-1}(X)\to X$ is a canonical desingularization
of $X.$ In particular every automorphism of $X$ has a lift to an
automorphism of $h^{-1}(X),$ i.e., $Aut(X)\subset
Aut(h^{-1}(X))=Aut(Y\setminus h^{-1}(S)).$ Since $V$ is a big
divisor, so is its pullback $h^*(V).$

Let $Z:=Y\setminus h^{-1}(X).$ It is a closed subvariety of $Y$.
Let $J_Z$ be the ideal sheaf of $Z$ and let $f: Y'\to Y$ be a
canonical principalization of $J_Z$ ( see e.g., \cite{kol},
\cite{wlo}). Thus $D:=f^{-1}(Z)$ is a hypersurface, which contains
a big hypersurface $V'=Supp(f^*h^*(V)).$ Since $D=V'+E$, where $E$
is effective divisor, the hypersurface $D$ is also big by Lemma
\ref{big}.

Finally if we take $g=f\circ h : Y'\to \overline{X}$, then
conditions 1) and 2) of Definition \ref{good} are satisfied.
\end{proof}

\section{The Quasi Minimal Model}

In this section following \cite{p-s}, we introduce the notion of
quasi-minimal models (for details see \cite{p-s}). This is a
weaker analog of a usual notion of minimal models which has an
advantage that to prove its existence we do not need the full
strength of the Minimal Model Program.

\begin{defi} (see \cite{p-s}) An effective $\Bbb Q$-divisor $M$
on a variety $X$ is said to be $Q$-movable if for some $n > 0$ the
divisor $nM$ is integral and generates a linear system without
fixed components.  Let X be a projective variety with $\Bbb
Q$-factorial terminal singularities. We say that X is a
quasi-minimal model if there exists a sequence of $\Bbb Q$-movable
$\Bbb Q$-divisors $Mj$ whose limit in the Neron-Severi space
$NSW_{\Bbb Q} (X) = NSW(X) \otimes \Bbb Q$ is $K_X.$
\end{defi}

By the recent progress in the minimal model program ( see
\cite{bir}, \cite{bir1}, \cite{p-s}) every non-uniruled smooth
variety has a quasi-minimal model. In fact if we ran MMP on $X$
and we do all possible divisorial contractions (and all necessary
flips) we achieve a quasi-minimal model $Y$, moreover we obtain
the mapping $\phi: X-\to Y$ which is a composition of divisorial
contractions ans flips, in particular it is a birational
contraction, i.e., the mapping $\phi^{-1}$ does not contract any
divisor (cf. \cite{p-s}, section 4, Corollary 4.5):

\begin{theo}\label{min}
Let $X$ be a smooth projective non-uniruled variety. Then there is
a quasi-minimal model $Y$ and a birational contraction $\phi: X
-\to Y.$$\square$
\end{theo}

Quasi minimal models have the following very important property
(cf. \cite{p-s}, section 4, Proposition 4.6):

\begin{theo}\label{min}
Let $X$ be a quasi-minimal model. Then $Bir(X)=Aut_1(X).$$\square$
\end{theo}

\section{Main Result.}

Now we can start our proof. The first step is

\begin{pr}\label{finite}
Let $X$ be a normal  complete non-uniruled variety and $H$ be a
big hypersurface on $X$. Then the group $Stab_X(H)=\{ f\in
Aut_{1}(X): f^*H=H \}$ is finite.
\end{pr}

\begin{proof}
For some $m\in \Bbb N$ the divisor $mD$ is very big. We have
$Stab_X(H)=\{ f\in Aut_{1}(X): f^*H=H \}=Stab_X(mH)=\{ f\in
Aut_{1}(X): f^*(mH)=mH \}.$

By the assumption the variety $X'=\overline{i_{mH}(X)}$ is
birationally equivalent to $X$. In view of Corollary \ref{natural}
it is enough to prove that the group $Lin(X')$ is finite.  Since
$X$ is non-uniruled, the $X'$ is non-uniruled, too. But the group
$Lin(X')$ is an affine group and if it is infinite, then  by
Rosenlicht Theorem (see \cite{ros}) we  have that $X'$ is ruled -
which is impossible.
\end{proof}

Now we can prove our main result:

\begin{theo}
Let $X$ be an open variety with a good cover.  If the group
$Aut(X)$ is infinite, then $X$ is uniruled.
\end{theo}

\begin{proof}
Assume that $Aut(X)$ is infinite.  Let $f: \overline{Y}\to
\overline{X}$ be a good cover of $X$ and take $Y=f^{-1}(X).$ Then
$Aut(Y)$ is also infinite. We have to prove that $X$ is uniruled.
To do this it suffices  to prove that $Y$ is uniruled.

Assume that $Y$ is not uniruled. By Theorem \ref{min} there exists
a quasi-minimal model $Z$ and a birational contraction $\phi:
\overline{Y}-\to Z.$ Take $\psi= \phi^{-1}.$ The mapping $\psi$ is
a regular mapping outside some closed subset $F$ of codimension
$\ge 2.$ By the Zariski Main Theorem the mapping $\psi$ restricted
to $Z\setminus F$ is an embedding.

Take a mapping $G\in Aut(Y),$ in fact $G\in Bir(\overline{Y}).$
The mapping $G$ induces a birational mapping $g\in Bir(Z).$ Since
$Bir(Z)=Aut_1(Z)$ we have $g\in Aut_1(Z).$ The mapping $g$ is a
morphisms outside a closed subset $R$ of codimension $\ge 2$,
moreover  sets $\overline{g^{-1}(R)}=R_1$ and
$\overline{g^{-1}(F)}=F_1$ have also codimension at least two ($g$
does not contract divisors).

The mapping $\psi$ embeds the set $V:=Z\setminus (F\cup F_1\cup
R\cup R_1)$ into $Y$, denote $U:= \psi(Z\setminus (F\cup F_1\cup
R\cup R_1)).$ Under this identification the mapping $g$ restricted
to $V$ corresponds to the mapping $G$ restricted to $U.$

Let $D=\overline{Y}\setminus Y$ be a big hypersurface, as in the
definition of a good cover. The hypersurface $D':=\psi^*(D)$ is
also big ( see Lemma \ref{big1}) and $D'\cap V$ corresponds to
$D\cap U.$ Since $G(U\setminus D)=U\setminus D$ we have that $g$
transforms irreducible components of $D'\cap V$ onto irreducible
components of $D'\cap V.$ In particular $g^*(D')=D'.$ This means
that $Aut(Y) \subset Stab_{D'}(Z)\subset Aut_1(Z).$ By Proposition
\ref{finite} this contradicts our assumption.
\end{proof}

\begin{cor}
Let $X$ be a quasi-affine (in particular affine) variety.  If the
group $Aut(X)$ is infinite, then $X$ is uniruled.
\end{cor}

\end{document}